\documentclass[12pt]{article}
\usepackage{amsmath,amsthm,amsfonts,amssymb,times,hyperref,graphicx,caption}
\usepackage{xcolor}

\newtheorem{theorem}{Theorem}
\newtheorem{lemma}{Lemma}

\theoremstyle{definition}

\newtheorem{conjecture}{Conjecture}

\DeclareMathOperator{\mmod}{\:mod}

\def\a{n}
\def\r{k}
\def\c{w}

\allowdisplaybreaks

\begin{document}

\centerline{\textit{\Large On the irreducibility of the non-cyclotomic part}}
\smallskip
\centerline{\textit{\Large of most $0,1$-polynomials with few terms}}

\medskip
\bigskip
\centerline{by M.~Filaseta and A.~Kalogirou}

\bigskip
\bigskip
\centerline{\textit{Dedicated to Mel Nathanson and Carl Pomerance in celebration of their $80^{\rm th}$ birthdays}}
\bigskip

\section{Introduction}

In \cite{sch7}, A.~Schinzel established the following interesting result.

\begin{theorem}[Schinzel, 1986]\label{schinzelthm}
Let $k$ be an integer $\ge 2$.
Let $a_{0}, a_{1}, \ldots, a_{k}$ be fixed non-zero integers.  Then for almost all choices of positive integers $n_{1} < \cdots < n_{k}$, the polynomial $F(x) = a_{0} + a_{1} x^{n_{1}} + \cdots + a_{k} x^{n_{k}}$, removed of its cyclotomic factors, is irreducible over $\mathbb Q$.
\end{theorem}

In other words, for a fixed integer $k \ge 2$, the number of choices for $n_{1}, \ldots, n_{k}$ in the interval $[1,N]$ for which $F(x)$, removed of its cyclotomic factors, is irreducible divided by $\binom{N}{k}$ tends to $1$ as the positive integer $N \ge k$ tends to infinity.
Notably, Schinzel's result is much stronger, allowing the coefficients to come from an arbitrary algebraic number field and giving an explicit upper bound on the number of exceptional choices of $n_{1}, \ldots, n_{k}$ in the interval $[1,N]$ for which $F(x)$, removed of its cyclotomic factors, is reducible.  This upper bound was improved upon in Schinzel's work \cite{sch12} to $\ll N^{\lfloor (k+1)/2 \rfloor}$, where the implied constant depends only on the coefficients $a_{0}, a_{1}, \ldots, a_{k}$.

Recently, the second author \cite{kal} made use of Theorem~\ref{schinzelthm} to establish, in an analogous lacunary setting, a conjecture of A.~Odlyzko and B.~Poonen \cite{OP}.   Odlyzko and Poonen conjectured that almost all polynomials with coefficients coming from $\{ 0, 1 \}$ and constant term $1$ are irreducible.  To help explain the second author's result, we note that one can show that a positive proportion, depending on $k$, of the polynomials of the form
\begin{equation}\label{F(x)eq}
F(x) = 1 + x^{n_{1}} + \cdots + x^{n_{k}},
\qquad
\text{where $n_{j} \in \mathbb Z^{+}$ and $n_{1} < \cdots < n_{k}$},
\end{equation}
are reducible; for example, a positive proportion will be divisible by $(x^{k+1}-1)/(x-1)$.  In \cite{kal}, the second author was able to show that this proportion must go to $0$ as $k$ goes to infinity.  More precisely, he establishes
\[
\lim_{k \rightarrow \infty} \limsup_{N \rightarrow \infty}
\dfrac{\big| \big\{ (n_{1}, \ldots, n_{k}) \in [1,N]^{k} : n_{1} < \cdots < n_{k},  F(x) \text{ in \eqref{F(x)eq} is reducible} \big\} \big|}{\binom{N}{k}} = 0.
\]

The Odlyzko-Poonen conjecture was resolved under the extended Riemann hypothesis by E.~F.~Breuillard and P.~P.~Varj\'u~\cite{BV}.  
The conjecture has, since it was proposed, sparked an interest in the study of polynomials with random coefficients and high degree, with the first breakthrough coming from S.~V.~Konyagin~\cite{Sergei}, whose work was highly influential both in the work of Breuillard and Varj\'u~\cite{BV}, as well as in the broader work of L.~Bary-Soroker and G.~Kozma~\cite{Bary} and L.~Bary-Soroker, D.~Koukoulopoulos and G.~Kozma~\cite{BKK}.
The conjecture remains open unconditionally, and while the question addressed in \cite{kal} has some external similarity with the Odlyzko-Poonen conjecture, the two are associated with very different random polynomial models and the two are independent.

In the next three sections, we provide a more elementary, largely self-contained, account of what is needed from Theorem~\ref{schinzelthm} for the second author's work in \cite{kal}, with the hope of making Schinzel's important work in \cite{sch7,sch12}, at least in the case of $0,1$-polynomials (polynomials whose coefficients are each $0$ or $1$), more easily accessible to a wider audience.  Specifically, we prove the following.

\begin{theorem}\label{schinzelweakthm}
Fix an integer $k \ge 2$.  
For almost all choices of positive integers $n_{1} < \cdots < n_{k}$, the polynomial $F(x) = 1 + x^{n_{1}} + \cdots + x^{n_{k}}$, removed of its cyclotomic factors, is irreducible.  More precisely, for $N$ a positive integer, the number of choices for $(n_{1}, \ldots, n_{k}) \in [1,N]^{k}$ with $n_{1} < \cdots < n_{k}$ and with $F(x)$, removed of its cyclotomic factors, reducible is $o(N^{k})$ as $N \rightarrow \infty$.
\end{theorem}

\noindent
In the last section of the paper, we also connect some of the ideas that permeate Schinzel's work and their presentation here with the conjecture posed by Odlyzko and Poonen.  We provide there an additive number theoretic  conjecture whose resolution would mark significant progress towards the Odlyzko-Poonen conjecture and is at the same time interesting in its own right.

To prove Theorem~\ref{schinzelweakthm}, we will take advantage of the fact that the coefficients of $F(x)$ are $0$ and $1$ in a way that will simplify what was done in \cite{sch7, sch12}.  
In particular, restricting the coefficients to $0$ and $1$ allows us to showcase more easily the themes around non-reciprocal factors, that is, factors $w(x)$ where $w(x) \ne \pm \,x^{\deg w} w(1/x)$.  Generally speaking, a lacunary polynomial will have more than one non-reciprocal irreducible factor if and only if its exponents have extra additive structure that one can exploit to reduce their count.  

To address possible irreducible reciprocal factors of $F(x)$ that are not cyclotomic, that is non-cyclotomic irreducible polynomials $w(x) \in \mathbb Z[x]$ of degree at least $1$ for which $w(x) = \pm \,x^{\deg w} w(1/x)$, our argument will borrow ideas from \cite{sch7} while providing some simplifications and different approaches to aspects of Schinzel's arguments there.

Before leaving this introduction, we address the important achievement of E.~Bombieri and U.~Zannier as given in the appendix by Zannier in \cite{schbook} (see also \cite[Theorem~1.6]{bmz} and \cite[Theorem~BZ]{sch12}).  We restrict the result to the context of Theorem~\ref{schinzelthm} and Theorem~\ref{schinzelweakthm}.

\begin{theorem}[Bombieri and Zannier, 1998]\label{bzthm}
Fix an integer $k \ge 2$.  
Let $n_{1} < \cdots < n_{k}$ be positive integers, and let 
$G(x_{1},\ldots,x_{k})$ and $H(x_{1},\ldots,x_{k})$ in $\mathbb Z[x_{1},\ldots,x_{k}]$ be relatively prime polynomials in $\mathbb Q[x_{1},\ldots,x_{k}]$.
Suppose $G(x^{n_{1}},\ldots,x^{n_{k}})$ and $H(x^{n_{1}},\ldots,x^{n_{k}})$ have a non-cyclotomic irreducible factor in common.  Then there are $\gamma_{j} \in \mathbb Z$, not all $0$, and a computable constant $C(G,H)$ only depending on $G$ and $H$ such that
\[
\gamma_{1} n_{1} + \cdots + \gamma_{k} n_{k} = 0 
\qquad \text{and} \qquad
\sum_{j=1}^{k}|\gamma_{j}| \le C(G,H).
\]
\end{theorem}

A proof of Theorem~\ref{bzthm} is outside the scope of this paper, but it is worth noting that Theorem~\ref{bzthm} allows us to quickly give an appropriate bound on the number of polynomials of the form given in Theorem~\ref{schinzelthm} which have an irreducible reciprocal non-cyclotomic factor.  Here, the terminology \textit{reciprocal}, as indicated above, has the following meaning.  If $f(x) \in \mathbb R[x]$, then the reciprocal of $f(x)$ is $\tilde{f}(x) = x^{\deg f}f(1/x)$.  A polynomial $f(x) \in \mathbb R[x]$ is said to be reciprocal if $f(x) = \pm \tilde{f}(x)$.  This definition is motivated by the fact that if $f(x)$ is reciprocal and $\alpha \in \mathbb C$ is a root of $f(x)$, then $\alpha \ne 0$ and $1/\alpha$ is also a root of $f(x)$.  The roots $\alpha$ and $1/\alpha$ will also occur in the factorization of $f(x)$ over $\mathbb C$ to the same multiplicity.  The cyclotomic polynomials are examples of reciprocal polynomials.

Fix non-zero integers $a_{0}, a_{1}, \ldots, a_{k}$.   
Take
\[
G(x_{1},\ldots,x_{k}) = a_{0} + a_{1} x_{1} + \cdots + a_{k} x_{k}
\quad \text{ and } \quad
H(x_{1},\ldots,x_{k}) = x_{1} \cdots x_{k} G(x_{1}^{-1},...,x_{k}^{-1}).
\]
One can check that these polynomials are relatively prime in $\mathbb Q[x_{1},\ldots,x_{k}]$ since they are distinct linear polynomials in $x_{1}$ (or in any other variable) and neither has a factor in $\mathbb Q[x_{2},\ldots,x_{k}]$.  Now, suppose 
$F(x) = a_{0} + a_{1} x^{n_{1}} + \cdots + a_{k} x^{n_{k}}$, as in the statement of Theorem~\ref{schinzelthm}, has an irreducible non-cyclotomic reciprocal factor $w(x) \in \mathbb Z[x]$, and let $\alpha$ be a root of $w(x)$.  Then $\alpha$ and $1/\alpha$ are roots of $w(x)$ and, hence, of $F(x)$.  We deduce that $\alpha$ is a root of both $G(x^{n_{1}},\ldots,x^{n_{k}}) = F(x)$ and
\[
H(x^{n_{1}},\ldots,x^{n_{k}}) 
= x^{n_{1}+\cdots+n_{k}} G(x^{-n_{1}},...,x^{-n_{k}})
= x^{n_{1}+\cdots+n_{k-1}} F(1/x).
\]
Thus, $G(x^{n_{1}},\ldots,x^{n_{k}})$ and $H(x^{n_{1}},\ldots,x^{n_{k}})$ are both divisible by $w(x)$.  
We obtain from Theorem~\ref{bzthm} that $n_{1}, \ldots, n_{k}$ must satisfy one of finitely many linear equations in $n_{1}, \ldots, n_{k}$ over $\mathbb Q$, where the linear equations depend only on $a_{0}, a_{1}, \ldots, a_{k}$ and in particular not on $n_{1}, \ldots, n_{k}$.  For each such equation, one of the variables $n_{j}$ can be expressed in terms of the others leading to $\ll N^{k-1}$ choices for $n_{1} < \cdots < n_{k} \le N$ satisfying the equation.  
We deduce then that the number of choices of positive integers $n_{1}, \ldots, n_{k}$, with $n_{1} < \cdots < n_{k} \le N$, for which $F(x)$ has an irreducible non-cyclotomic reciprocal factor is $\ll_{k} N^{k-1}$ which is small compared to $\binom{N}{k}$ and thus negligible in comparison to all $F(x)$ considered in Theorem~\ref{schinzelthm}.  
To clarify, as in Schinzel's initial work on this topic in \cite{sch7}, we will not make use of this application of Theorem~\ref{bzthm} as we proceed.


\section{Non-reciprocal factors}\label{secttwononrecippart}

The \textit{non-reciprocal part} of a polynomial $f(x) \in \mathbb Z[x]$ is $f(x)$ removed of every irreducible reciprocal factor in $\mathbb Z[x]$ having positive leading coefficient.  Thus, for example, the non-reciprocal part of $f(x) = 2 x (x+1)^{2} (3x^2 - 3x - 1) (-2x^2+3x-2)$ is $- x (3x^2 - 3x - 1)$, which corresponds to the irreducible factors $2$, $x+1$ twice, and $2x^{2}-3x+2$ in $\mathbb Z[x]$ being removed from $f(x)$.  
In this section, we show that for almost all polynomials $F(x)$ as in Theorem~\ref{schinzelweakthm}, the non-reciprocal part of $F(x)$ is irreducible.  

\begin{lemma}\label{NRlemma}
Fix an integer $k \ge 2$. For all but $O_k(N^{k-1})$ choices of positive integers $n_1 < n_2 < \cdots < n_k \le N$, the $0,1$-polynomial $F(x)$ in \eqref{F(x)eq} has exactly one irreducible non-reciprocal factor. 
\end{lemma}

Faithful to the themes throughout our treatment on $0,1$-polynomials, we will arrive at the above conclusion by finding linear equations satisfied by the components of the vector $\vec{n} = \langle n_1, \ldots, n_k \rangle$, consisting of the non-zero exponents of $F(x)$, in the case where $F(x)$ does not have exactly one irreducible non-reciprocal factor.  The idea that will produce these linear constraints is contained in the following lemma from \cite{fil1999}.  We give a proof to emphasize the simplicity of this result.

\begin{lemma}\label{fillemma}
 Let $f(x) \in \mathbb Z[x]$ be a $0,1$-polynomial, with $f(0) = 1$ and such that $f(x)$ is divisible by the product of two non-reciprocal irreducible polynomials in $\mathbb Z[x]$, not necessarily distinct. 
 Then there is a $0,1$-polynomial $g(x)$ with the same number of terms as $f(x)$ and with $g(x)\not\in \{f(x),\tilde{f}(x)\}$ satisfying
    \begin{equation}\label{non-recip}
        f(x)\tilde{f}(x)=g(x)\tilde{g}(x).
    \end{equation}  
\end{lemma}

\begin{proof}
We begin by showing that there are monic non-reciprocal polynomials $w_1(x)$ and $w_2(x)$ in $\mathbb Z[x]$ such that $f(x)=w_1(x)w_2(x)$. 
Since the non-zero coefficients of $f(x)$ are positive and $f(0) = 1$, every monic factor of $f(x)$ in $\mathbb Z[x]$ cannot have positive real roots and therefore must have constant term $1$.  Therefore, for some positive integers $\lambda_i, \lambda'_i$ and $\lambda''_j$, we can write 
\[
f(x)=\prod_{i=1}^{t_1} p^{\lambda_i}_i(x){\tilde{p}}^{\lambda'_i}_i(x)\prod_{j=1}^{t_2} q^{\lambda''_j}_j(x)\prod_{e=1}^{t_3} r_e(x),
\]
where in the first product $p_i(x)$ and $\tilde{p}_i(x)$ are pairs of monic irreducible non-reciprocal factors of $f(x)$ in $\mathbb Z[x]$, in the second product $q_j(x)$ are monic irreducible non-reciprocal factors of $f(x)$ in $\mathbb Z[x]$ for which $\tilde{q}_j(x)$ is not a factor of $f(x)$, and in the third product the $r_e(x)$ are monic irreducible reciprocal factors of $f(x)$ in $\mathbb Z[x]$ which are not necessarily distinct. 
Furthermore, we take, as we can, the $2t_1+t_2$ polynomials $p_i(x)$, $\tilde{p}_i(x)$ and $q_j(x)$ to be distinct.  
If the first product is empty, then we can take $w_1(x)$ to be any irreducible factor of the second product, which will have at least two factors by assumption. If the first product is not empty, then we can take $w_1(x)$ to be $\prod_{i=1}^{t_1} p^{\lambda_i}_i(x)$.

Set $g(x)=w_1(x)\tilde{w}_2(x)$, where $f(x)=w_1(x)w_2(x)$ and where $w_1(x)$ and $w_2(x)$ are monic and non-reciprocal. Then $g(x)\not\in \{f(x),\tilde{f}(x)\}$, and $g(x)$ satisfies \eqref{non-recip}. 
As noted above, since $w_1(x)$ and $w_2(x)$ are monic factors of $f(x)$, we have $w_1(0) = w_2(0) = 1$.  It follows that $g(x)$ is monic.  
It remains to show that $g(x)$ is a $0,1$-polynomial with the same number of terms as $f(x)$. 

Write $g(x)=\sum_{i=1}^l a_ix^{m_i}$ with integers $0 \le m_0 < m_1 < \cdots < m_l$,  $a_i\neq 0$, and $a_l = 1$. Since $f(0) = 1$, so that $x$ is not a factor of $f(x)$, we see that 
\[
n = \deg f = \deg \tilde{f} = \deg g = \deg \tilde{g}.
\]
Let $k$ be the number of terms of $f(x)$. By comparing the coefficients of $x^n$ on both sides of \eqref{non-recip}, we find that $\sum_{i=1}^la_i^2=k,$ which implies that $l\leq k.$ Using \eqref{non-recip} again, we find that $k^2=f(1)\tilde{f}(1)=g(1)\tilde{g}(1).$ Together with the Cauchy-Schwarz inequality, we obtain
\[
k^2=f(1)\tilde{f}(1)=g(1)\tilde{g}(1) = \left(\sum_{i=1}^l a_i\right)^{2} \leq l\left(\sum_{i=1}^l a_i^2\right)=lk\leq k^2.  
\]
This proves that $l=k,$ and the equality in the Cauchy-Schwarz inequality implies that $g(x)$ has all its coefficients equal to 1.
\end{proof}

Let $F(x)$ be as in \eqref{F(x)eq} with $k \ge 2$, and let $G(x)$ be a $0,1$-polynomial with the same number of terms as $F(x)$ satisfying
\begin{equation}\label{constraint}
    F(x)\tilde{F}(x)=G(x)\tilde{G}(x).
\end{equation}
Since $G(x) = F(x)$ satisfies \eqref{constraint}, we know such a $G(x)$ exists.  We show that for the majority of the exponent vectors $\vec{n}$, the equation \eqref{constraint} implies that $G(x)\in\{F(x),\tilde{F}(x)\}$. Lemma~\ref{fillemma} will then imply that the majority of the vectors $\vec{n}$ are such that either the non-reciprocal part of $F(x)$ is irreducible or the non-reciprocal part of $F(x)$ is $1$. It will then remain to bound the number of vectors $\vec{n}$ such that the non-reciprocal part of $F(x)$ is $1$.

From \eqref{F(x)eq} and \eqref{constraint}, we have $F(0) = G(0) = 1$.  
Write $F(x)=\sum_{i=0}^k x^{n_i}$ and $G(x)=\sum_{i=0}^kx^{m_i}$, where $0 = n_0 < n_1 < \cdots < n_k \le N$ and $0 = m_0 < m_1 < \cdots < m_k$. Set $d = n_k$. From \eqref{constraint}, we deduce that 
\begin{equation}\label{firstdiff}
m_k=n_k=d.
\end{equation}
We then see that 
\[
\tilde{F}(x)=\sum_{i=0}^{k} x^{d-n_i}
\qquad \text{and} \qquad
\tilde{G}(x)=\sum_{i=0}^kx^{d-m_i},
\]
where here the monomials show up in descending degree order.

With the above notation, 
we have that the second largest degree of a  monomial on the left-hand side of \eqref{constraint} is either $2d-n_1$ or $n_{k-1}+d$. 
Observe that if these two values are equal, then $n_1 + n_{k-1} = d = n_k$, and the total number of vectors $\vec{n} = \langle n_1, \ldots, n_k \rangle$ satisfying this equation is at most $N^{k-1}$. The quantity $N^{k-1}$ is small compared to the total number of possibilities for $F(x)$ as in \eqref{F(x)eq}, which is $\binom{N}{k}$ with $k$ fixed.
So we restrict now to the case that $2d-n_1 \ne n_{k-1}+d$.  
We will use this idea again, that if we have a linear combination of the $n_1, \ldots, n_k$, with coefficients not all $0$, with the linear combination equal to $0$, then the vectors $\vec{n}$ satisfying this equation account for $O(N^{k-1})$ of the $\binom{N}{k}$ different $F(x)$ in \eqref{F(x)eq}.

Next, we note that we can suppose that the second largest degree of a monomial on the left-hand side of \eqref{constraint} is $2d-n_1$. 
Otherwise, we can redefine the $n_i$ by interchanging the roles of $F(x)$ and $\tilde{F}(x)$ so that equation \eqref{F(x)eq} still holds and the second largest degree of a monomial on the left-hand side of \eqref{constraint} is $2d-n_1$, as desired.  Similarly, by interchanging the roles of $G(x)$ and $\tilde{G}(x)$ if need be, we can suppose that the second largest degree of a monomial on the right-hand side of \eqref{constraint} is $2d-m_1$. 
Doing so, we deduce then that $n_{1} = m_{1}$ or, equivalently, that
\begin{equation}\label{secdiff}
n_k-n_1=m_k-m_1.
\end{equation}
The goal is to show that for all but $O_k(N^{k-1})$ vectors $\vec{n}$, the remaining $m_k-m_i$, with $2\leq i\leq k-1$, satisfy
\[
m_k-m_i = n_k-n_i.
\]
This will then demonstrate, since \eqref{firstdiff} holds, that $G(x)=F(x)$.

Suppose there is some $i \in \{ 2, \ldots, k-1 \}$ such that $m_k-m_i \ne n_k-n_i$. Since the product on the right-hand side of \eqref{constraint} will include a term of degree $m_k-m_i$ which also must appear on the left-hand side of \eqref{constraint}, we deduce there exist $(g,h) \ne (k,i)$ such that $g > h$ and
\begin{equation}\label{thirddiff}
    m_k-m_i=n_g-n_h.
\end{equation}
Here, it is possible that $n_h=n_0=0$. We consider the possibility that $g\neq k$ in \eqref{thirddiff}.  Either $n_h\neq n_1$ or $n_h\neq n_0=0$. If $n_h\neq n_1,$ then taking the difference of \eqref{secdiff} and \eqref{thirddiff}, we find that
\begin{equation}\label{anothereq}
m_i-m_1 = n_k+n_h-n_g-n_1.
\end{equation}
Here, we have $k > g > h$ and $n_1 \not\in \{ n_h, n_k \}$, so the right-hand side of \eqref{anothereq} is a linear combination of at least $3$ distinct $n_i$. 
There is some term on the right-hand side of \eqref{constraint} of degree $d + m_i - m_1$ which must equal a term on the left-hand side of \eqref{constraint} of degree $d + n_u - n_v$ for some $u$ and $v$, so that the left-hand side of \eqref{anothereq} can be replaced by a difference in two components of $\vec{n}$. As noted earlier, the resulting equation in the $n_i$ implies there are $O_{k}(N^{k-1})$ vectors $\vec{n}$ for which $n_h\neq n_1$ in \eqref{thirddiff}.

With still $g\neq k$, we suppose now $n_h \ne n_0$ or, equivalently, $h \ne 0$.
In this case, we take the difference of \eqref{firstdiff} and \eqref{thirddiff}, leading to \eqref{anothereq} but with $m_1$ and $n_1$ replaced by $m_0$ and $n_0$, respectively.  Proceeding as we did for $n_h \ne n_1$, we deduce again that there are $O_{k}(N^{k-1})$ vectors $\vec{n}$ for which $n_h\neq n_0$ in \eqref{thirddiff}.

We can deduce now that for all but $O_k(N^{k-1})$ vectors $\vec{n}$, if \eqref{thirddiff} holds for some $i \in \{ 2, \ldots, k-1 \}$ and some $g$ and $h$, then we must have $g = k$. Recalling \eqref{firstdiff} and \eqref{secdiff}, we deduce then that, except for $O_k(N^{k-1})$ exceptions of the $\vec{n}$, for each $m_k - m_i$ with $i \in \{ 0, 1, \ldots, k-1 \}$, there is an $h \in \{ 0, 1, \ldots, k-1 \}$ such that $m_k - m_i = n_k - n_h$.  The inequalities
\[
m_k - m_0 > \ldots > m_k - m_{k-1}
\qquad \text{and} \qquad
n_k - n_0 > \ldots > n_k - n_{k-1}
\]
now imply that, outside $O_k(N^{k-1})$ exceptions, we have $m_i = n_i$ for every $i \in \{ 0, 1, \ldots, k-1 \}$, implying $G(x) = F(x)$. 

Lemma~\ref{fillemma} implies now that for all but $O_k(N^{k-1})$ of the vectors $\vec{n}$,  either the non-reciprocal part of $F(x)$ is irreducible or the non-reciprocal part of $F(x)$ is $1$. It remains to bound the number of vectors $\vec{n}$ such that the non-reciprocal part of $F(x)$ is $1$.  However, if the non-reciprocal part of $F(x)$ is $1$, then $F(x)$ is necessarily reciprocal so that $n_{k} - n_{k-1} = n_{1}$. Thus, the non-reciprocal part of $F(x)$ is $1$ for at most $O(N^{k-1})$ vectors $\vec{n}$.  Lemma~\ref{NRlemma} follows.


\section{Preliminary information for analyzing reciprocal factors}

For an integer $n$ and a positive integer $d$, we use the notation $n \mmod d$ to refer to the unique integer $a \in \{ 0,1,\ldots,d-1 \}$ such that $n \equiv a \pmod{d}$.  
The following lemma is a variation of a lemma in \cite{ffk} (see also \cite{fil2023}).

\begin{lemma}\label{ffklemma}
Let $\r$ and $N$ be positive integers.  
Let $\varepsilon = \varepsilon(\r,N) \in (0,1/2]$.  
Set  $\kappa = \lfloor 1/\varepsilon \rfloor + 1$ and
\[
V(\r) = \max \bigg\{ \dfrac{\kappa^{2\r-2}}{2} \bigg( \varepsilon - \dfrac{1}{\kappa}  \bigg)^{-1}, 
 \kappa^{\r-1} + \varepsilon \bigg\}.
\]  
If $\a_{1}, \a_{2}, \dots, \a_{\r}$ are nonnegative integers satisfying
$\a_{1} < \a_{2} < \cdots < \a_{\r} \le N$ and $\a_{\r} \ge V(\r)$, then 
there exists an integer  
\begin{equation}\label{ffkpap2}
d \in \bigg( \dfrac{\a_{\r}}{\kappa^{\r-1} + \varepsilon}, \dfrac{\a_{\r}}{1 - \varepsilon} \bigg)
\end{equation}
such that 
$\a_{j} \mmod d$ is in $[0,\varepsilon d) \cup ((1-\varepsilon) d,d)$
for every $j \in \{1,2,\dots,\r\}$.
\end{lemma}

\begin{proof}
For $\r = 1$, we have $\a_{\r} \ge V(\r) \ge \kappa^{\r-1} + \varepsilon > 1$, so we can take $d = \a_{\r}$.  
We consider now the case that $\r > 1$.   
Since $\varepsilon \le 1/2$, we have $\kappa = \lfloor 1/\varepsilon \rfloor + 1 \ge 3$.
By the definition of $\kappa$, we have $\kappa > 1/\varepsilon$ or, equivalently, $\varepsilon > 1/\kappa$.  
The condition $\a_{\r} \ge V(\r)$ implies that
\begin{equation}\label{dirichletboxlemeq1}
\dfrac{\a_{\r}}{\kappa^{\r-1} + \varepsilon} \ge 1 \qquad \text{ and } \qquad
\sqrt{2 \a_{\r}} \,\bigg( \varepsilon - \dfrac{1}{\kappa} \bigg)^{1/2} \ge \kappa^{\r-1}.
\end{equation}
Note that the conditions in the lemma and $\r > 1$ imply that $d$ exists as in \eqref{ffkpap2} since, by the first inequality in \eqref{dirichletboxlemeq1}, the interval in \eqref{ffkpap2} has length 
\[
\dfrac{\a_{\r} (\kappa^{\r-1} + 2\varepsilon - 1)}{(\kappa^{\r-1} + \varepsilon)(1 - \varepsilon)}
\ge \dfrac{\kappa + 2\varepsilon - 1}{1 - \varepsilon}
\ge  2\cdot \dfrac{1 + \varepsilon}{1 - \varepsilon} > 1.
\]
To establish the lemma, it suffices to show that
there is an integer $d$ satisfying \eqref{ffkpap2} and integers $\c_{1}, \c_{2}, \dots, \c_{\r}$ 
such that
\begin{equation}\label{ffkpap3}
|\a_{j} - \c_{j} d| < \varepsilon d \qquad \text{ for } 1 \le j \le \r. 
\end{equation}

Let $x_{j} = \a_{j}/\a_{\r}$ for $j \in \{1,2,\dots,\r\}$. 
We explain next how the
Dirichlet box principle implies that there is a positive integer $t \le \kappa^{\r-1}$ satisfying  
\begin{equation}\label{ffkpap4}
\Vert t x_{j} \Vert < \dfrac{1}{\kappa} \qquad \text{ for } 1 \le j \le \r-1,
\end{equation}
where $\Vert x \Vert$ denotes the distance from $x$ to the nearest integer.
For each $t' \in \{ 1, 2, \dots, \kappa^{\r-1}+1 \}$, we define the
point 
\[
P(t') = \big( \{ t'x_{1} \}  , \{ t'x_{2} \}, \dots, \{ t'x_{\r-1} \} \big),
\] 
with coordinates in the interval $[0,1)$.  
For each of the $\kappa^{\r-1}$ choices of $u_{j} \in \{ 0, 1, \ldots, \kappa -1 \}$, where $1 \le j \le \r-1$, 
we consider the cube
\begin{align*}
\mathcal C(u_{1}, \dots, u_{\r-1}) = 
\{ (s_{1}, s_{2}, \dots, s_{\r-1}) : u_{j}/\kappa \le s_{j} < &(u_{j}+1)/\kappa \\ &\text{ for } 1 \le j \le \r-1  \}.
\end{align*}
By the Dirichlet box principle, there are $t_{1}'$ and $t_{2}'$ in 
$\{ 1, 2, \ldots, \kappa^{\r-1}+1 \}$
with $t_{1}' > t_{2}'$
such that the points $P(t_{1}')$ and $P(t_{2}')$ are both in the same cube
$C(u_{1}, \dots, u_{\r-1})$.  Set 
$t = t_{1}'-t_{2}' \in \{ 1, 2, \ldots, \kappa^{\r-1} \}$.
For every $j \in \{ 1, 2, \dots, \r-1 \}$, we obtain
\[
t x_{j} = t_{1}'x_{j} - t_{2}'x_{j} 
= \lfloor t_{1}'x_{j} \rfloor - \lfloor t_{2}'x_{j} \rfloor 
+ \{ t_{1}'x_{j} \} - \{ t_{2}'x_{j} \}.
\]
It follows that each such $t x_{j}$ is within 
$|\{ t_{1}'x_{j} \} - \{ t_{2}'x_{j} \}| < 1/\kappa$
of the integer $\lfloor t_{1}'x_{j} \rfloor - \lfloor t_{2}'x_{j} \rfloor$.  
This establishes \eqref{ffkpap4}.  
Since $x_{k} = 1$, we see that \eqref{ffkpap4} holds with $j = \r$ as well.  
Observe that \eqref{dirichletboxlemeq1} implies that 
\begin{equation}\label{dupperbd}
t \le \sqrt{2 \a_{\r}} \,\bigg( \varepsilon - \dfrac{1}{\kappa} \bigg)^{1/2}.
\end{equation}
For $1 \le j \le \r$, take $\c_{j}$ to be the nearest
integer to $t x_{j}$. 

For the moment, suppose $\c_{j} \ne 0$ (so that $\c_{j} \ge 1$) for each 
$j \in \{1,2,\dots,\r\}$.  Then \eqref{ffkpap3} follows provided
\[
\dfrac{d}{\a_{\r}} \in \bigcap_{1 \le j \le \r} 
\bigg( \dfrac{x_{j}}{\c_{j} + \varepsilon}, \dfrac{x_{j}}{\c_{j} - \varepsilon} \bigg).
\]
Recall $\varepsilon > 1/\kappa$.  
For each $j \in \{1,2,\dots,\r\}$, since $t x_{j} \le t$ so that $\c_{j} \le t$, we deduce from \eqref{ffkpap4} that
\[
\dfrac{t+(1/\kappa)}{t+\varepsilon} \ge \dfrac{\c_{j}+(1/\kappa)}{\c_{j}+\varepsilon} 
> \dfrac{t x_{j}}{\c_{j}+\varepsilon}
\]
and
\[
\dfrac{t-(1/\kappa)}{t-\varepsilon} \le \dfrac{\c_{j}-(1/\kappa)}{\c_{j}-\varepsilon} 
< \dfrac{t x_{j}}{\c_{j}-\varepsilon}.
\]
Hence, \eqref{ffkpap3} holds provided
\begin{equation}\label{ffkpap5}
\dfrac{d}{\a_{\r}} \in 
\bigg( \dfrac{t+(1/\kappa)}{t(t+\varepsilon)}, \dfrac{t-(1/\kappa)}{t(t-\varepsilon)} \bigg).
\end{equation}
Using \eqref{dupperbd}, we see that the length of the interval on the right is 
\[
\dfrac{2}{t^{2}-\varepsilon^{2}} \bigg( \varepsilon - \dfrac{1}{\kappa} \bigg) 
> \dfrac{2}{t^{2}} \bigg( \varepsilon - \dfrac{1}{\kappa} \bigg) \ge \dfrac{1}{\a_{\r}}
\]
so that $d$ exists satisfying \eqref{ffkpap5}.  
From \eqref{ffkpap5} and $t \in \{ 1, 2, \ldots, \kappa^{\r-1} \}$, we see that
\[
\dfrac{d}{\a_{\r}} > \dfrac{t+(1/\kappa)}{t(t+\varepsilon)} > \dfrac{1}{t+\varepsilon} \ge \dfrac{1}{\kappa^{\r-1} +\varepsilon}
\]
and
\[
\dfrac{d}{\a_{\r}} < \dfrac{t-(1/\kappa)}{t(t-\varepsilon)} < \dfrac{1}{t-\varepsilon} \le \dfrac{1}{1-\varepsilon}.
\]
Thus, we also see that \eqref{ffkpap2} holds.  

Now, suppose some $\c_{j} = 0$ with $j \in \{1,2,\dots,\r\}$.  We again choose
$d$ so that \eqref{ffkpap5} holds and, hence, \eqref{ffkpap2} as well.  For each $\c_{j} \ne 0$, the above argument gives 
$|\a_{j}-\c_{j}d| < \varepsilon d$ as in \eqref{ffkpap3}.  On the other hand, if $\c_{j} = 0$, then
\eqref{ffkpap4} and the definitions of $\c_{j}$, $x_{j}$, and $d$ imply
\[
\kappa t \a_{j} < \a_{\r} < d 
\dfrac{t \big( t + \varepsilon \big)}{t + (1/\kappa)}
\le \kappa \varepsilon t d.
\]
Hence, \eqref{ffkpap3} holds for such $\c_{j}$ as well, completing the proof.
\end{proof}

The above lemma serves the same purpose as Lemma~1 in \cite{sch7}, where instead Schinzel makes use of work of E.~Bombieri and J.~D.~Vaaler \cite{bombvaaler} and of E.~Hlawka \cite{hlawka}.  

We will want to make use of the \textit{Newton polytope} (or \textit{Newton polyhedron}) associated with a multivariate polynomial, so we briefly discuss this topic next. 
Let $x_1,\dots,x_r$ be indeterminates, and let $P(x_1,\dots,x_r)\in K[x_1,\dots,x_r]$ be a polynomial in $r$ variables over a field $K$. We write
\[
P(x_1,\dots,x_r) = \sum_{i=1}^{s} \alpha_{i} \prod_{j=1}^{r} x_j^{c_{ij}},
\]
where $\alpha_i \in K\backslash \{ 0 \}$ and $c_{i,j} \in \mathbb Z^{+} \cup \{ 0 \}$.  Define the points $p_{i} = (c_{i1},\dots,c_{ir})$, for $1 \le i \le s$.  The \textit{Newton polytope associated with $P(x_1,\dots,x_r)$} is the convex hull of the points $p_{i}$, which can be written as
\[
\mathcal N(P)=\left\{ \sum_{i=1}^{s} t_i p_i : t_i \in [0,1], \sum_{i=1}^{s}t_i=1 \right\},
\]
where scalar multiplication and addition of points is performed coordinate-wise in the same way as vectors.  
In the case of two variables where $r = 2$, the boundary of the convex hull $\mathcal N(P)$ is referred to as the \textit{Newton polygon} associated with $P(x_1,\dots,x_r)$.  

A classical result for Newton polytopes due to A.~M.~Ostrowski~\cite{ostrowski} is the following (cf. Theorem~18, p.~89, of \cite{schbook}).

\begin{lemma}[Ostrowski, 1975]\label{ostrowskithm}
If $P(x_1,\dots,x_r)$ and $Q(x_1,\dots,x_r)$ are two non-zero polynomials in $K[x_1,\dots,x_r]$, then
\[
\mathcal N(P Q) = \mathcal N(P) + \mathcal N(Q),
\] 
where the sum on the right is the \textit{Minkowski sum} defined by
\[
\mathcal N(P) + \mathcal N(Q) = \{ p + q : p \in \mathcal N(P), q \in \mathcal N(Q) \}.
\]
\end{lemma} 

Next, we provide another well-known result pertaining to the Minkowski sum of two convex polytopes in $\mathbb R^{2}$ (see for example Theorem~13.5 and its proof in \cite{compgeometry}).  
To describe the result, it is helpful to give some terminology for the edges (sides) of a polytope in $\mathbb R^{2}$.  
We say an edge of a polytope is a lower (an upper) edge if the polytope lies on or above (on or below) the line formed by extending that edge.  
It is also possible that an edge is a vertical line segment, and we say an edge is a left vertical (a right vertical) edge if the polytope lies on or to the right of (on or to the left of) the line formed by extending that edge.  
Finally, in the context of Newton polytopes in $\mathbb R^{2}$, it is possible that a polytope is a line segment. 
In this case, we view the polytope as having two edges, a lower edge and an upper edge (or a left vertical edge and a right vertical edge).  
Note that some types of edges may not exist for a given polytope.  For example, if a polytope is just a vertical line segment, then it has a left vertical edge and a right vertical edge but does not have lower or upper edges.  
The following result describes the Minkowski sum of two polytopes in $\mathbb R^{2}$ in terms of translates of the edges of the two polytopes.

\begin{lemma}\label{compgeom}
If $P(x_{1},x_{2})$ and $Q(x_{1},x_{2})$ are two non-zero polynomials in $K[x_{1},x_{2}]$, then the boundary of the Minkowski sum
$\mathcal N(P) + \mathcal N(Q)$ can be formed by translating the edges of $\mathcal N(P)$ and $\mathcal N(Q)$ in such a way that the lower (upper, left vertical, right vertical) edges along $\mathcal N(P)$ and $\mathcal N(Q)$ are translated to form the lower (upper, left vertical, right vertical, respectively) edges along $\mathcal N(P) + \mathcal N(Q)$ without overlapping and with the slopes of the edges of $\mathcal N(P) + \mathcal N(Q)$ increasing (decreasing, undefined, undefined, respectively) going from left to right.  
\end{lemma} 

To clarify, if a lower (upper, left vertical, right vertical) edge of length $\ell_{1}$ on $\mathcal N(P)$ is parallel to a lower (upper, left vertical, right vertical) edge of length $\ell_{2}$ on $\mathcal N(Q)$, then they are translated to form a lower (upper, left vertical, right vertical) edge of length $\ell_{1} + \ell_{2}$ on $\mathcal N(P) + \mathcal N(Q)$.  Below we illustrate Lemmas~\ref{ostrowskithm} and \ref{compgeom} with an example of two Newton polytopes and their Minkowski sum.

\begin{figure}[h!]
        \centering
        \includegraphics[height=5cm]{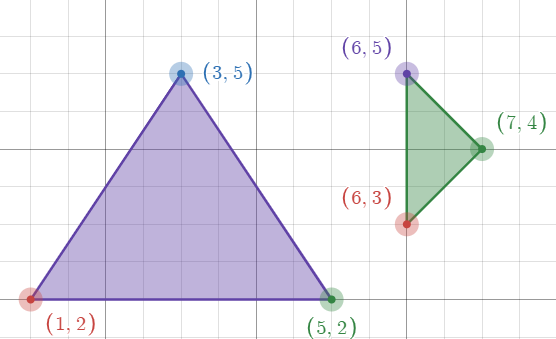}
        \qquad
        \includegraphics[height=5cm]{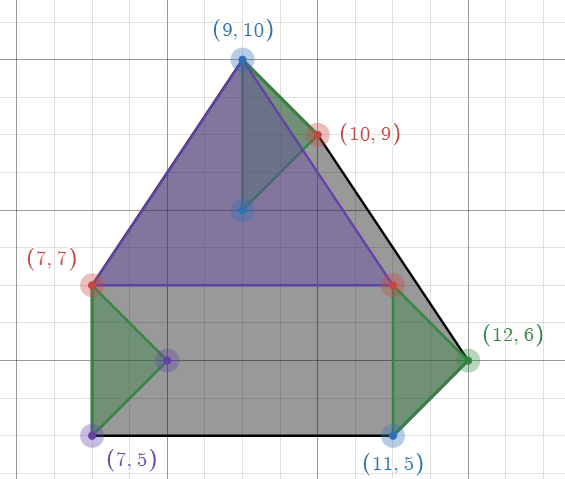}
        \caption*{Figure:  \parbox[t]{10cm}{The Newton Polytopes for $P(x,y) = xy^2+x^5y^2+x^3y^5$, $Q(x,y) = x^6y^3+x^7y^4+x^6y^5$ and $P(x,y) \cdot Q(x,y)$}}
        \label{fig:enter-label}
\end{figure}


\section{Finishing the proof of Theorem~\ref{schinzelweakthm}}

As a result of Lemma~\ref{NRlemma}, to complete the proof of Theorem~\ref{schinzelweakthm}, we will show that there are $o(N^k)$ polynomials $F(x)$ as in \eqref{F(x)eq} with $n_k \le N$ for which $F(x)$ has an irreducible non-cyclotomic reciprocal factor.  
We take $\r \ge 2$ to be fixed and view $N$ as sufficiently large compared to $\r$.
The conditions in Lemma~\ref{ffklemma} require $V(\r) \le n_{\r} \le N$ so that, in particular, to use the lemma, we will want
\begin{equation}\label{commenteq}
N \ge V(\r) = \max \bigg\{ \dfrac{\kappa^{2\r-2}}{2} \bigg( \varepsilon - \dfrac{1}{\kappa}  \bigg)^{-1}, 
 \kappa^{\r-1} + \varepsilon \bigg\}.
\end{equation}
A hurdle to overcome is how small the difference of $\varepsilon$ and $1/\kappa$ is as this difference appears on the right-hand side of \eqref{commenteq}. 
We take $\varepsilon = 1/\big(\lfloor N^{1/(2\r+1)} \rfloor - 1\big)$.  
Then 
\begin{equation}\label{kappabd}
\kappa = \lfloor 1/\varepsilon \rfloor + 1 = \lfloor N^{1/(2\r+1)} \rfloor \sim N^{1/(2\r+1)},
\end{equation}
as $N \rightarrow \infty$, and we obtain
\[
\varepsilon - \dfrac{1}{\kappa} = \dfrac{1}{\big( \lfloor N^{1/(2\r+1)} \rfloor - 1 \big) \lfloor N^{1/(2\r+1)} \rfloor}
\ge \dfrac{1}{\big( N^{1/(2\r+1)} - 1 \big) N^{1/(2\r+1)}} > \dfrac{1}{N^{2/(2\r+1)}}.
\]
It follows that
\[
\dfrac{\kappa^{2\r-2}}{2} \bigg( \varepsilon - \dfrac{1}{\kappa}  \bigg)^{-1}
< \dfrac{\lfloor N^{1/(2\r+1)} \rfloor^{2\r-2}}{2} N^{2/(2\r+1)}
\le \dfrac{N^{2\r/(2\r+1)}}{2} < N^{2\r/(2\r+1)}.
\]
Since $\r \ge 2$ is fixed and $N$ is sufficiently large, we see that
\[
\kappa^{\r-1} + \varepsilon 
\le \lfloor N^{1/(2\r+1)} \rfloor^{\r-1} + \dfrac{1}{\big(\lfloor N^{1/(2\r+1)} \rfloor - 1\big)} 
\le \sqrt{N} < N^{2\r/(2\r+1)}.
\]
Thus, with $\varepsilon = 1/\big(\lfloor N^{1/(2\r+1)} \rfloor - 1\big)$, we obtain that $V(\r) < N^{2\r/(2\r+1)}$ and \eqref{commenteq} follows.
With $\varepsilon$ as above, we discard from consideration the $\le V(k)^{k} < N^{2k^2/(2\r+1)} = N^{\r} N^{-\r/(2\r+1)}$ vectors $\vec{\a}$ for which $\a_{\r} < V(\r)$, noting that there are $o(N^k)$, as $N \rightarrow \infty$, such vectors $\vec{\a}$.  

We consider now vectors $\vec{\a}$ with $1 \le \a_{1} < \cdots < \a_{\r} \le N$ and $\a_{\r} \ge V(\r)$.
By Lemma~\ref{ffklemma}, with $k$ fixed and $N$ large, there is a positive integer $d$ and integers $m_{j}$ and $t_{j}$ such that 
\begin{equation}\label{nmdteq}
\a_{j} = m_{j} d + t_{j}
\qquad \text{and} \qquad
|t_{j}| < \varepsilon d < \dfrac{1.01 d}{N^{1/(2\r+1)}}, 
\end{equation}
for all $j \in \{1, \ldots, k\}$.
Also, from \eqref{ffkpap2} and our choice of $\varepsilon$, we choose, as we may, $d$ so that
\begin{equation}\label{dbounds}
\dfrac{n_k}{N^{(k-1)/(2k+1)}} \ll \dfrac{\a_{\r}}{\kappa^{\r-1} + \varepsilon} < d < 1.01 n_k.
\end{equation}
Define
\[
m' = \max_{1 \le j \le k}\{ m_{j} \} = m_{\r}, \quad
t' =  \max\{ 0, t_{1}, \ldots, t_{k} \} \quad \text{and} \quad
t'' = - \min\{ 0, t_{1}, \ldots, t_{k} \}.
\]
Since $k$ is fixed and $N$ is large, we deduce from \eqref{kappabd}, \eqref{nmdteq} and \eqref{dbounds} that
\[
m' < \dfrac{\a_{\r}}{d} + 1 < \kappa^{\r-1} + 1 + \varepsilon \ll \sqrt{N},
\]
\[
\max\{ t',t'' \} < \dfrac{1.01 d}{N^{1/(2\r+1)}} < \dfrac{1.03 N}{N^{1/(2\r+1)}} = 1.03 \,N^{2\r/(2\r+1)},
\]
and
\[
m' \cdot \max\{ t',t'' \} < \bigg( \dfrac{N}{d} + 1 \bigg) \dfrac{1.01 d}{N^{1/(2\r+1)}} < 2.04 \,N^{2\r/(2\r+1)}.
\]
Following a key idea of Schinzel's in \cite{sch7}, we define
\begin{equation}\label{GHdef}
\begin{split}
G(y,z) &= y^{t''} + \sum_{j = 1}^{k} y^{t_{j}+t''} z^{m_{j}} \in \mathbb Z[y,z], \\
H(y,z) &= y^{t'} z^{m'} + \sum_{j = 1}^{k} y^{t'-t_{j}} z^{m'-m_{j}} \in \mathbb Z[y,z].
\end{split}
\end{equation}
A significant motivation here is that $G(x,x^{d}) = x^{t''}F(x)$ and $H(x,x^{d}) = x^{a} \tilde{F}(x)$, where $a$ is a nonnegative integer and $\tilde{F}(x) = x^{\deg F} F(1/x) = x^{n_{k}} F(1/x)$.
In particular, if $g(x) \in \mathbb Z[x]$ is a monic reciprocal factor of $F(x)$, then $g(x)$ is a factor of both $G(x,x^{d})$ and $H(x,x^{d})$.  

For a polynomial $P(y,z) \in \mathbb Z[y,z]$, let $\deg_{y} P$ and $\deg_{z} P$ denote the degree of $P(y,z)$ in $y$ and the degree of $P(y,z)$ in $z$, respectively.  Momentarily, we will make use of the inequality
\begin{equation}\label{degresultbd}
\big( \deg_{y} G \big) \big( \deg_{z} H \big) + \big( \deg_{z} G \big) \big( \deg_{y} H \big) 
\le 2 (t' + t'') m' < 9 \,N^{2\r/(2\r+1)}.
\end{equation}

We are now in a position to prove the following lemma, which is a simplified version of Lemma~2 from \cite{sch7}.

\begin{lemma}\label{lemma2}
Let $F(x)$ be as in \eqref{F(x)eq} with $k \ge 3$ and $V(k) \le n_k \le N$. If $g(x)$ is an irreducible reciprocal factor of $F(x)$ in $\mathbb Z[x]$, then either
\begin{equation}\label{smalldegree}
    \deg g < 9 N^{2k/(2k+1)}
\end{equation}  
or there exist four distinct nonnegative indices $i,j,u,v\le k$ such that 
\begin{equation}\label{gencommonfac2}
    \begin{vmatrix}
        n_i-n_j&m_i-m_j\\n_u-n_v&m_u-m_v
    \end{vmatrix}=0.
\end{equation}
\end{lemma}

\smallskip
\begin{proof}
Note that we can take $g(x)$ to be monic since $F(x)$ is monic.  
Let 
\[
D(y,z) = \gcd(G(y,z),H(y,z)),
\] 
where the greatest common divisor is taken over $\mathbb Z[y,z]$.  One can justify that this is the same as the greatest common divisor taken over $\mathbb Q[y,z]$, but this is not needed.  We consider two cases, depending on whether $D(y,z)$ is a constant or not.  Since $g(x)$ is a monic reciprocal factor of $F(x)$, we have $g(x)$ is a factor of both $G(x,x^{d})$ and $H(x,x^{d})$. 
Thus, either $g(x) \mid D(x,x^{d})$ or $g(x) $ divides both $G(x,x^d)/D(x,x^{d})$ and $H(x,x^d)/D(x,x^{d})$.

We consider first the case that $D(y,z)$ is a constant.  For a root $\xi$ of $g(x)$, we see that
\[
G(\xi,\xi^d)=0=H(\xi,\xi^d).
\]
Since $D(y,z)$ is a constant, the resultant of $G(y,z)$ and $H(y,z)$, with respect to $z$, say, is a non-zero polynomial $R(y)$. Furthermore, there exist two non-zero polynomials $A(y,z)$ and $B(y,z)$ in $\mathbb Z[y,z]$ such that
\[
A(y,z)G(y,z)+B(y,z)H(y,z) = R(y).
\]
The substitution $y = \xi$ and $z = \xi^{d}$ shows that $\xi$ is a root of $R(y)$ so that $\deg_y R$ is a bound on $\deg g$.  The definition of the resultant (or its equivalent form as the determinant of a Sylvester matrix) together with \eqref{degresultbd} implies
\[
\deg_y R \le \big( \deg_{y} G \big) \big( \deg_{z} H \big) + \big( \deg_{z} G \big) \big( \deg_{y} H \big) 
< 9 \,N^{2\r/(2\r+1)}.
\]
We deduce that \eqref{smalldegree} holds.

Now, we consider the case that $D(y,z)$ is not a constant.  We appeal to Lemmas~\ref{ostrowskithm} and \ref{compgeom}.  Since $D(y,z)$ is a factor of $G(y,z)$ and $H(y,z)$, these lemmas imply that either a left vertical edge or a lower edge (whichever exists, possibly both), say $\mathcal E$, of $\mathcal N(D)$ can be translated to form part (or all) of an edge of a left vertical or a lower edge of $\mathcal N(G)$ and of $\mathcal N(H)$. On the other hand, the definition of $G(y,z)$ and $H(y,z)$ in \eqref{GHdef} implies that $\mathcal N(G)$ is a translation of the reflection of $\mathcal N(H)$ about the origin.  This reflection takes a left vertical edge or a lower edge of $\mathcal N(H)$ to a parallel right vertical or parallel upper edge on the reflection of $\mathcal N(H)$.  We deduce then that $\mathcal N(G)$ consists of two edges parallel to $\mathcal E$, either a left and right vertical edge or a lower and upper edge. The idea is to deduce then that we can find $4$ distinct points along the edges of $\mathcal N(G)$, which we call $(t_{i},m_{i}), (t_{j},m_{j}), (t_{u},m_{u})$ and $(t_{v},m_{v})$, such that the edge through $(t_{i},m_{i})$ and $(t_{j},m_{j})$ is parallel to the edge through $(t_{u},m_{u})$ and $(t_{v},m_{v})$. What needs a little bit of care for this deduction is the case that $\mathcal N(G)$ is a line segment. However, as a condition in the lemma, we have $k \ge 3$ which in turn implies from \eqref{GHdef} that $\mathcal N(G)$ is the convex hull of at least $4$ distinct points arising from the degrees of $y$ and $z$ in the monomials in $G(y,z)$. If $\mathcal N(G)$ is a line segment, then we can deduce that there are at least $4$ points on this line segment, and labeling them as $(t_{i},m_{i}), (t_{j},m_{j}), (t_{u},m_{u})$ and $(t_{v},m_{v})$, we arrive at the same conclusion as above, that the edge through $(t_{i},m_{i})$ and $(t_{j},m_{j})$ is parallel to the edge through $(t_{u},m_{u})$ and $(t_{v},m_{v})$.  We deduce 
\[
\begin{vmatrix}
    t_i-t_j& m_i-m_j\\ t_u-t_v&m_u-m_v
\end{vmatrix}
    =0.
\] 
For each $e \in \{ i,j,u,v \}$, from \eqref{nmdteq}, we have $t_e = n_e - m_e d$, and now \eqref{gencommonfac2} follows.
\end{proof}

The above lemma will be used to restrict the possibilities for the vector $\vec{n}$ when $k \ge 3$. For the purposes of establishing the main result in \cite{kal}, there is no harm in restricting to $k \ge 3$.  However, we can refer to the literature to resolve the remaining case where $k = 2$ in Theorem~\ref{schinzelweakthm}. In this case, the polynomial $F(x)$ is a $0,1$-polynomial having $3$ terms, and independent work of W.~Ljunggren~\cite{ljunggren} and H.~A.~Tverberg~\cite{tverberg} implies that $F(x)$ removed of its cyclotomic factors is irreducible or $1$.  In the latter case, the non-reciprocal part of $F(x)$ is $1$, and we already showed at the end of Section~\ref{secttwononrecippart} that there are at most $O(N^{k-1})$ such $F(x)$.  This then is sufficient for addressing the case that $k = 2$ in Theorem~\ref{schinzelweakthm}.

We begin taking advantage of Lemma~\ref{lemma2} by making use of a recent result of V.~Dimitrov~\cite{dimitrov} (see also J.~F.~McKee and C.~J.~Smyth~\cite{mckeesmyth} and W.~Zudilin~\cite{zudilin}) resolving the Schinzel-Zassenhaus conjecture.  This will allow us to address the possibility that an irreducible non-cyclotomic reciprocal factor $g(x)$ of $F(x)$ satisfies \eqref{smalldegree}.  Before doing so, however, we point out that one could obtain sufficient information for handling the case of \eqref{smalldegree} by instead using a result of E.~Dobrowolski \cite{dobrowolski}, which is how Schinzel addressed $g(x)$ having a similar degree bound in \cite{sch7}.  The choice of using Dimitrov's result is for convenience (the argument is a little shorter) and to emphasize the utility of this excellent work.

Dimitrov's result in \cite{dimitrov} implies that if $g(x)$ is an irreducible non-cyclotomic polynomial satisfying \eqref{smalldegree}, then there is a root $\xi$ of $g(x)$ satisfying 
\begin{equation}\label{dimitrovbd}
|\xi|\ge 2^{1/(4b)}, 
\qquad \text{where $b=9 N^{2k/(2k+1)} \ge \deg g$}.
\end{equation}
This leads to the following.

\begin{lemma}\label{mahlerarg}
    If $F(x)$ has an irreducible non-cyclotomic factor $g(x)$ satisfying \eqref{smalldegree}, then 
    \begin{equation}\label{lemmadiffbd}
        n_k-n_{k-1}\leq \frac{36\log k}{\log 2}N^{2k/(2k+1)}.
    \end{equation}
\end{lemma}

\begin{proof}
Since $g(x)$ is a factor of $F(x)$, there is a root $\xi$ of $g(x)$ and, hence, $F(x)$ satisfying
\[
\xi^{n_k}=-\xi^{n_{k-1}}-\cdots-1, 
\]
where $\xi$ is as in \eqref{dimitrovbd}.  
By the triangle inequality, we obtain
\[
|\xi|^{n_k}\leq k|\xi|^{n_{k-1}}.
\]
Taking logarithms of both sides, we see that
\[
        n_k-n_{k-1} \le \frac{\log k}{\log |\xi|}
        \le \frac{4b \log k}{\log 2}
        =\frac{36\log k}{\log 2}N^{2k/(2k+1)},
\]
completing the proof.
\end{proof}

The conditions on $g(x)$ in Lemma~\ref{lemma2} and Lemma~\ref{mahlerarg} are different, but we are interested in their commonality, where $g(x)$ is an irreducible non-cyclotomic reciprocal factor of $F(x)$.  We note that Lemma~\ref{mahlerarg} is the crucial place where we are taking into account that $g(x)$ is not cyclotomic. Other aspects of the arguments in this section apply even if $g(x)$ is cyclotomic.   

We immediately see that \eqref{lemmadiffbd} happens for a vanishing proportion of vectors $\vec{n}$; there are at most $N^{k-1}$ choices for the first $n_1,n_2\dots,n_{k-1}$ coordinates and then $\ll_{k} N^{2k/(2k+1)}$ for $n_k$. In particular, there are at most
$O_k\big( N^{k - (1/(2k+1))}\big)$ different $F(x)$ for which an irreducible non-cyclotomic reciprocal factor $g(x)$ can satisfy \eqref{smalldegree}.  
Furthermore, an analogous argument allows us to say that for all but $O_k\big( N^{k - (1/(2k+1))}\big)$ different $F(x)$ we have that no two of $n_1, \ldots, n_k$ differ in absolute value by $\ll_k N^{2k/(2k+1)}$ as $N$ goes to infinity.  Here, we have used that there are at most $\binom{k}{2}$ choices for two different $n_j$ and the implied constant in the $\ll_k$ symbol can be an arbitrarily large constant depending only on $k$.  We will make more use of this bound on the number of $n_1, \ldots, n_k$ with no two differing in absolute value by $\ll_k N^{2k/(2k+1)}$ momentarily, but for now we deduce from \eqref{nmdteq} and \eqref{dbounds} that for all but $O_k\big( N^{k - (1/(2k+1))}\big)$ polynomials $F(x)$ as in \eqref{F(x)eq} with $n_k \le N$, the values of $m_1, \ldots, m_k$ are distinct.  More specifically, if some $m_i = m_j$ with $1 \le i < j \le k$, then \eqref{nmdteq} and \eqref{dbounds} imply
\[
0 < n_j - n_i \le |t_j| + |t_i| < \dfrac{2.02 d}{N^{1/(2k+1)}} \le \dfrac{3 N}{N^{1/(2k+1)}} \ll N^{2k/(2k+1)},
\]
which occurs for $O_k\big( N^{k - (1/(2k+1))}\big)$ of the $F(x)$.

We restrict now to the case that $m_1, \ldots, m_k$ are distinct and \eqref{gencommonfac2} holds.  As we also know that $n_1, \ldots, n_k$ are distinct, we obtain that the $4$ entries on the left in \eqref{gencommonfac2} are all non-zero.  From \eqref{gencommonfac2}, we see that
\begin{equation}\label{nmdiff}
(n_{i}-n_{j}) (m_{u}-m_{v}) = (n_{u}-n_{v}) (m_{i}-m_{j}).
\end{equation}
We deduce now from \eqref{nmdteq} and \eqref{dbounds} that
\[
0 < |m_{i}-m_{j}| \le \max\{ |m_{i}|, |m_{j}| \} 
< \dfrac{n_{k}}{d} + 1
< \dfrac{3 n_{k}}{d} \ll N^{(k-1)/(2k+1)}.
\]
The same lower and upper bounds hold on $|m_{u}-m_{v}|$.  
On the other hand, we see from \eqref{nmdiff} that $(n_{i}-n_{j})/\gcd(n_{i}-n_{j},n_{u}-n_{v})$ divides $m_{i}-m_{j}$ and $(n_{u}-n_{v})/\gcd(n_{i}-n_{j},n_{u}-n_{v})$ divides $m_{u}-m_{v}$.  Given the bounds on $|m_{i}-m_{j}|$ and $|m_{u}-m_{v}|$, we deduce that
\begin{equation}\label{maxgcd}
    \frac{\max\{ |n_i-n_j|,|n_u-n_v| \}}{\gcd(|n_i-n_j|,|n_u-n_v|)} \ll N^{(k-1)/(2k+1)}.
\end{equation}
From \eqref{maxgcd}, we see that either 
\begin{equation}\label{gcdbd}
\gcd(|n_i-n_j|,|n_u-n_v|) \ge N^{(k+1)/(2k+1)}
\end{equation}
or
\begin{equation}\label{maxbd}
\max\{ |n_i-n_j|,|n_u-n_v| \} \ll N^{2k/(2k+1)}.
\end{equation}
In the case of \eqref{maxbd}, we have $|n_i-n_j| \ll N^{2k/(2k+1)}$ so that, by our earlier comments, we deduce there are at most $O_k\big( N^{k - (1/(2k+1))}\big)$ polynomials $F(x)$ as in \eqref{F(x)eq} with $n_k \le N$ and \eqref{maxbd} holding for some $i, j, u$ and $v$ as above.  In the case of \eqref{gcdbd}, we see that there are at most $N^{k-2}$ choices for 
\[
\{ n_1, \ldots, n_k \}\backslash\{ n_{j},n_{v} \},
\]
and, given $n_k \le N$, there are $\le N$ choices for $\gcd(|n_i-n_j|,|n_u-n_v|)$. Given these, we deduce from \eqref{gcdbd} that there are $\ll N/N^{(k+1)/(2k+1)} \ll N^{k/(2k+1)}$ choices for $n_j$ and $\ll N^{k/(2k+1)}$ choices for $n_v$.  Hence, there are at most $O_k\big( N^{k - (1/(2k+1))}\big)$ polynomials $F(x)$ as in \eqref{F(x)eq} with $n_k \le N$ and \eqref{gcdbd} holding for some $i, j, u$ and $v$ as above.  

In conclusion, there are $o(N^k)$ polynomials as in Lemma~\ref{lemma2} with $g(x)$ an irreducible non-cyclotomic reciprocal polynomial.  This completes the proof of Theorem~\ref{schinzelweakthm}.


\section{A conjecture associated with Lemma~\ref{fillemma}}

It would be of some interest to establish, in the direction of the conjecture of Odlyzko and Poonen \cite{OP}, that almost all of the $2^{N}$ polynomials of the form \eqref{F(x)eq}, with $0 \le k \le N$ and $n_{k} \le N$, have an irreducible non-reciprocal part.  This would easily follow from Lemma~\ref{fillemma} if one can show the following.

\begin{conjecture}\label{ourconj}
Let $A \subseteq \{0,1,2,\dots,n\}$ with $0 \in A$. Consider the function $\phi$ defined for all such $A$, sending $A$ into the multiset with $|A|^{2}$ elements given by
\[
A-A=\{ x-y:x,y\in A  \}.
\]
Then the image of $\phi$ contains $2^{n-1} + o(2^{n})$ distinct elements as $n \rightarrow \infty$.
\end{conjecture}

Observe that if $A$ is as in the conjecture and $a'$ is the maximum element of $A$, then the set
\[
A' = \{ a' - a: a \in A \}
\]
is such that the multisets $A-A$ and $A'-A'$ are identical.  Furthermore, $A' \subseteq \{0,1,2,\dots,n\}$ with $0 \in A'$.  The conjecture is claiming that an element in the image of $\phi$ almost always determines the pair $A$ and $A'$.  
If we view $A$ as the set of exponents of $x$ in $F(x)$ (including $0$) in \eqref{F(x)eq}, then $A'$ corresponds to the set of exponents of $x$ in the reciprocal $\tilde{F}(x)$.

Lemma~\ref{fillemma} implies that Conjecture~\ref{ourconj} follows from the Odlyzko-Poonen conjecture.  
Conjecture~\ref{ourconj} is also true under the extended Riemann hypothesis, through the roundabout argument of linking it to the conditional resolution of the Odlyzko-Poonen conjecture by Breuillard and Varj\'u~\cite{BV}. 
The objective then would be to give an unconditional proof of Conjecture~1 and hopefully one that is elementary and more intrinsic to the conjecture.


\begin{thebibliography}{99}

\bibitem{BKK}
L.~Bary-Soroker, D.~Koukoulopoulos and G.~Kozma,
\textit{Irreducibility of random polynomials: general measures},
Invent.~Math.~233 (2023), 1041--1120. 

\bibitem{Bary}
L.~Bary-Soroker and G.~Kozma, \textit{Irreducible polynomials of bounded height}, Duke Math. J.~169 (2020), 579--598.

\bibitem{compgeometry}
M.~de Berg,  O.~Cheong,  M.~van Kreveld and M.~Overmars, 
Computational geometry, 
Algorithms and applications, third edition,
Springer-Verlag, Berlin, 2008.

\bibitem{bmz}
E.~Bombieri, D.~Masser and U.~Zannier, 
\textit{Anomalous subvarieties---structure theorems and applications}, Int.~Math.~Res.~Not.~IMRN (2007), no.~19, Art. ID rnm057, 33 pp.

\bibitem{bombvaaler}
E.~Bombieri and J.~D.~Vaaler, 
\textit{On Siegel's lemma}, 
Invent.~Math.~73 (1983), 11--32;
\textit{Addendum to: ``On Siegel's lemma''}, 
Invent.~Math.~75 (1984), p.~377.

\bibitem{BV}
E.~F.~Breuillard and P.~P.~Varj\'u, \textit{Irreducibility of random polynomials of large degree}, Acta Math.~223 (2019), 195--249.

\bibitem{dimitrov}
V.~Dimitrov, 
\textit{A proof of the Schinzel-Zassenhaus conjecture on polynomials},
arXiv: 1912.12545 [math.NT], 2019, 27pp.

\bibitem{dobrowolski}
E.~Dobrowolski, 
\textit{On a question of Lehmer and the number of irreducible factors of a polynomial}, 
Acta Arith.~34 (1979), 391--401.

\bibitem{fil1999}
Michael Filaseta,
\textit{On the factorization of polynomials with small {E}uclidean norm},
In {\rm Number theory in progress, Vol. 1
(Zakopane-Ko\'scielisko, 1997)}, de Gruyter, Berlin, 1999, 143--163.

\bibitem{fil2023}
M.~Filaseta,
\textit{On the factorization of lacunary polynomials}, 
Acta Arith.~210 (2023), 23--52.

\bibitem{ffk}
M.~Filaseta, K.~Ford, and S.~Konyagin,
\textit{On an irreducibility theorem of {A}.~{S}chinzel associated with coverings of the integers},
Illinois J. Math.~44 (2000), 633--643.

\bibitem{hlawka}
E.~Hlawka, \textit{\"Uber Gitterpunkte in Zylindern}, \"Osterreich.~Akad.~Wiss.~Math.-Nat.~Kl.~S.-B.~IIa 156 (1948), 203--217.

\bibitem{kal}
A.~Kalogirou,
\textit{Irreducibility of lacunary polynomials with 0,1 coefficients},
arXiv:2410.10035.

\bibitem{Sergei}
S.~V.~Konyagin, \textit{On the number of irreducible polynomials with $0,1$ coefficients}, Acta Arith.~88 (1999), 333--350. 

\bibitem{ljunggren}
W.~Ljunggren, \textit{On the irreducibility of certain trinomials and quadrinomials}, Math.~Scand.~8 (1960), 65--70.

\bibitem{mckeesmyth}
J.~F. McKee and C.~J. Smyth, Around the unit circle -- Mahler measure, integer matrices and roots of unity, Universitext, Springer, Cham, 2021.

\bibitem{OP}
A.~M.~Odlyzko and B.~Poonen, \textit{Zeros of polynomials with $0, 1$ coefficients},
Enseign.~Math.~39 (1993), 317--348.

\bibitem{ostrowski}
A.~M.~Ostrowski, 
\textit{On multiplication and factorization of polynomials. I. Lexicographic orderings and extreme aggregates of terms}, Aequationes Math.~13 (1975), 201--228.

\bibitem{sch7}
A.~Schinzel, 
\textit{Reducibility of lacunary polynomials VII}, Monatsh.~Math.~102 (1986), 309--337.

\bibitem{sch12}
A.~Schinzel, 
\textit{Reducibility of lacunary polynomials XII}, Acta Arith.~90 (1999), 273--289.

\bibitem{schbook}
A.~Schinzel, Polynomials with special regard to reducibility, Encyclopedia of Mathematics and its Applications, 77, Cambridge Univ. Press, Cambridge, 2000.

\bibitem{tverberg}
H.~A.~Tverberg, \textit{On the irreducibility of the trinomials $x\sp{n}\pm x\sp{m}\pm 1$}, Math.~Scand.~8 (1960), 121--126.

\bibitem{zudilin}
W.~Zudilin, Analytic methods in number theory -- when complex numbers count, Monographs in Number Theory, 11, World Sci.~Publ., Hackensack, NJ, 2024.


\end{thebibliography}
\end{document}